\documentclass[12pt]{article}
\usepackage{amsthm}
\usepackage{amsfonts}
\usepackage{amssymb}
\usepackage{amsmath}
\usepackage{nccmath}
\usepackage{mathtools}
\usepackage{mathrsfs}
\usepackage{setspace}
\usepackage{graphicx}
\usepackage[export]{adjustbox}
\usepackage{hyperref}
\usepackage[numbers,sort&compress]{natbib}
\usepackage{enumerate}
\usepackage{authblk}
\usepackage{placeins}
\usepackage{subcaption}

\newtheorem{theorem}{Theorem}[section]
\newtheorem{lemma}[theorem]{Lemma}

\theoremstyle{definition}
\newtheorem{definition}[theorem]{Definition}
\newtheorem{example}[theorem]{Example}

\theoremstyle{remark}
\newtheorem{remark}[theorem]{Remark}

\numberwithin{equation}{section}

\usepackage{amsmath,amssymb,graphicx} 
\usepackage[margin=1in]{geometry} 
\usepackage{enumerate}
\usepackage{authblk}
\usepackage{placeins}
\usepackage{array}
\usepackage{academicons}
\usepackage{xcolor}
\usepackage{svg}
\usepackage[utf8]{inputenc}
\usepackage{booktabs}
\usepackage{tikz}
\usepackage{placeins}
\usetikzlibrary{angles, arrows.meta, calc}
\usepackage{caption}
\usepackage{subcaption} 
\usepackage{float}
\usepackage{pgfplots}
\usetikzlibrary{intersections}
\newcolumntype{C}[1]{>{\centering\arraybackslash}p{#1}}
\newcommand{\head}[1]{\textbf{#1}} 
\usepackage{pgfplots}
\pgfplotsset{compat=1.16} 
\pgfplotsset{compat=1.16} 
\definecolor{lightblue}{rgb}{0.68, 0.85, 0.9} 
\tikzset{
  reversepath/.style={
    insert path={(current bounding box.south west) rectangle (current bounding box.north east)}
  }
}
\providecommand{\keywords}[1]{\textbf{\textit{Keywords:}} #1}
\providecommand{\subjclass}[1]{\textbf{\textit{MSC2020:}} #1}
\begin{document}

\nocite{*} 

\title{ Analysis of the  Distribution and  Asymptotic Approximations of Roots of the Polynomial Equation $$ z^{n+1}=(1+z)^n, n \in \mathbb{N}  $$ in the Complex Plane}

\author{Hailu Bikila Yadeta \\ email: \href{mailto:haybik@gmail.com}{haybik@gmail.com} }
  \affil{Salale University, College of Natural Sciences, Department of Mathematics\\ Fiche, Ethiopia}
\date{\today}
\maketitle

\noindent
\begin{abstract}
\noindent
 We study the spatial distribution of  the  positive, negative and non-real complex roots $z_n $  of  the sequence the $(n+1)$th degree polynomial equation
$$ z^{n+1}=(1+z)^n,\quad  n \in \mathbb{N}.$$
 We establish asymptotic approximations to  the sequence of the negative, the positive and the non-real complex roots of the equation as $n\rightarrow \infty $.
 In addition, we discuses the possible areas of applications of the current problem.
\end{abstract}

\noindent\keywords{Lambert's W function, Descartes's rule of sign, asymptotic approximation, golden ratio.
Intermediate value theorem, fundamental theorem of algebra }\\ 
\subjclass{Primary:  30C15, 30E99 }\\
\subjclass{Secondary: 41A60}

\section{Introduction}
In this paper, we study  the sequence of  polynomial equations:
\begin{equation} \label{eq:original}
z^{n+1} = (1 + z)^n,\quad  n \in\mathbb{ N }.
\end{equation}
We analyse the nature of the solutions to the equation,  for each $n \in\mathbb{ N} $, determining whether the roots are positive real numbers, negative real numbers, or complex conjugate pairs. In particular, we demonstrate that the equation has exactly one positive roots for each $n \in \mathbb{N} $. Furthermore, we show that for odd $n$  there is exactly one negative real root, whereas for even $n$, there are no negative roots. The practical applications and usability of the current study may include the following areas:

\begin{itemize}

\item The problem involves exploring the nature of roots of polynomial equations, incorporating applications of fundamental theorems such as the Fundamental Theorem of Algebra, Descartes’ Rule of Signs, the Intermediate Value Theorem, Rouch\'{e}'s Theorem and other concepts from complex analysis. This area of study spans various branches of mathematics, including the theory of equations, number theory, complex analysis, approximation theory, numerical analysis, and elementary calculus.

  \item  The equation (\ref{eq:original}) can be viewed as a generalization of the equation
  \begin{equation}\label{eq:GRequation}
     z^2 =1 + z,
  \end{equation}
  that corresponds to $n=1$. The positive root of this equation  is the golden ratio $\varphi$ and is given by
  $$ \varphi = \frac{1 +\sqrt{5}}{2}$$
  The golden ratio appears in different areas of mathematics including geometry. The negative root of the equation (\ref{eq:GRequation}) is given by $z= -\frac{1}{\varphi}$.  The fundamental relation
   $$\varphi ^2= \varphi +1 $$
   is sufficient to express any higher powers of $\varphi $ as a linear  expression in $\varphi$ .

  \item The investigation of the solutions' nature to the equation (\ref{eq:original}) can be relevant in analysing certain nonlinear recurrence relations. For instance, consider the first order \emph{nonlinear }recurrence relation:
      \begin{equation}\label{eq:nonlinearrecurrence}
       x_{n+1}= \left( 1+\frac{1}{x_n}\right)^n, \quad  n \in \mathbb{N}.
      \end{equation}
  The equilibrium point of the recurrence relation (\ref{eq:nonlinearrecurrence}) corresponds to  the fixed point of the function
  \begin{equation}\label{eq:fixedpointfunction}
     f(x) = \left( 1+\frac{1}{x}\right)^n .
  \end{equation}
   A fixed point of the function $f$  is any number $\bar{x}$ that satisfies the equation
      \begin{equation}\label{eq:fixedpointoff}
          x =  f(x).
      \end{equation}
Equilibrium solutions $x_n= \bar{x}$ are crucial  for analyzing the long term behavior of the solution to the iterative equation. Identifying the fixed points of the function (\ref{eq:fixedpointoff}) is equivalent to finding the solution of the equation (\ref{eq:original}).

\item We demonstrate that the positive roots of the equation (\ref{eq:original}) exist for each $n$ and form a strictly increasing sequence. While the closed-form formula for this sequence is not provided in this paper, we derive an asymptotic formula for $ x_n $ as $ n \to \infty $. To achieve this, we establish a connection with Lambert's $W$ function.

\item In the study of the $(n+1)$-th order \emph{linear} difference equation
      \begin{equation}\label{eq:corrospondingdifferenceequation}
         E^{n+1} y(t)= (E+I)^n y(t),
      \end{equation}
   where
   $$  E ^ry(t):= y(t+r) $$
    is a shift of $r$ units, and $ E^0 : = I $ is the identity operator. To solve the difference equation (\ref{eq:corrospondingdifferenceequation}),  we have to solve its characteristic equation, which is the same as the current equation (\ref{eq:original}). The  general  solution of the difference equation is determined by the roots of the characteristics equations. The study of the stability condition of the solution is bases on  the nature of  the roots.

    \item  In the study of the $(n+1)$th order \emph{linear} ordinary differential equation
      \begin{equation}\label{eq:corrospondingdifferentialequation}
        D^{n+1} y(t)= (D+I)^n y(t),
      \end{equation}
  where
   $$ D^ry(t)= \frac{d^r}{dt^r}y(t) =  y^{(r)} (t)  $$
   is the $r$-th order derivative of $y$.  By convention, the zeroth order derivative $D^0$  is considered   the identity operator $I$. The characteristic equation  of the differential equation (\ref{eq:corrospondingdifferentialequation}) is the polynomial equation (\ref{eq:original}). The roots of the characteristic equation are crucial in determining the general solution of the linear differential equation. In addition to the nature of the roots, such as whether the real parts are positive, zero, or negative, these characteristics play an important role in assessing the stability and long-term behavior of the solution.
\end{itemize}


\section{ Exploring the Roots of Equations by Regions in Complex Plane: Nonexistence, Existence, and Their Natures with Asymptotic Approximations}

\subsection{Positive roots of the equation}

\begin{lemma}[\textbf{Descartes' Rule of sign}]
  The number of positive roots of a polynomial $f(x)$ with real coefficients is equal to, the number of changes in  signs  of its coefficients (each root being counted the number of times equal to its multiplicity) minus a nonnegative even integer. The number of negative roots of $f(x)$ is equal to the number of changes in the signs of the coefficients of the polynomial $f(-x)$ minus a nonnegative even integer.
\end{lemma}

 \begin{theorem}
   For each $n\in \mathbb{N} $ the equation (\ref{eq:original}) has exactly one positive real root $x_n$.
 \end{theorem}
 \begin{proof}
    The equation (\ref{eq:original}) can written as:
   $$ z^{n+1}-\sum_{s=0}^{n}\binom{n}{s}z^s =0 .$$
   Therefore all coefficients, $ -\binom{n}{s}, s=0,1,...,n $, are negatives except the coefficient of  $z^{n+1}$  which is equal to $1$. The proof follows by the Descartes' rule of sign.
 \end{proof}

 \begin{theorem}
   Let $\{x_n\}$ is a sequence of the positive roots of the equation (\ref{eq:original}). Then $\{x_n\}$ is strictly increasing sequence.
 \end{theorem}
\begin{proof}
  Assume the negation of  the strictly increasing condition by claiming that there exists an $ N \in \mathbb{N }$ such that $ x_ N \geq  x_{N+1} > 0 $.  Then
  $$ x_{N+1}= \left( 1+ \frac{1}{x_{N+1}}\right)^{N+1} \geq \left( 1+ \frac{1}{x_{N}}\right)^{N+1}  >\left( 1+ \frac{1}{x_{N}}\right)^{N} =x_N . $$
  This contradicts the assumption $ x_ N \geq  x_{N+1} > 0 $. Therefore the sequence of the positive roots is strictly increasing.
\end{proof}
We proved that  the positive roots of the family of equations (\ref{eq:original}) form a strictly increasing sequence. It is good idea to investigate wether this sequence is bounded  or it tends to positive infinity. The next theorem answers this question.
\begin{theorem}
  Consider the sequence $\{x_n\}$ of positive roots of the equation(\ref{eq:original}). Then $x_n\rightarrow \infty $ as $n\rightarrow \infty $.
\end{theorem}
\begin{proof}
  Assume that $x_n$ is a positive root of (\ref{eq:original}). Suppose that there exists $M > 0 $  such that $0 < x_n \leq M $ for all $ n \in  \mathbb{N}$. Then,
  $$ x_n= \left(1+\frac{1}{x_n}\right)^n \geq \left(1+ \frac{1}{M} \right)^n , \quad \forall n \in \mathbb{N}. $$
  The last inequality contradicts with the assumption that  the sequence $\{x_n\}$ is bounded above, since for sufficiently  large $n$,  we have $(1+1/M) ^n  > M > x_n $.  This is a contradiction to the assumption that the sequence $\{ x_n\}$ is bounded above by $M$. It follows that  $x_n\rightarrow \infty $ as $n\rightarrow \infty $.
\end{proof}
Our next task is to derive an asymptotic approximation to the sequence of positive roots of the equation. Since no closed-form expression for these  roots is known, it is crucial to asses whether an asymptotic sequence can serve a good approximation for the positive terms.

\subsection{ Asymptotic approximation of the sequence of positive roots}
We have observed that the positive roots of the equation exist for all $n$ and form an increasing sequence that is not bounded. Since there is no closed-form formula for the sequence of positive roots, we focus on establishing an asymptotic approximation. To this end, we a connection with the asymptotic expansion of the well-known Lambert's $W$ function.
\begin{definition}
  The Lambert's $W $ function is defined as a multi-valued  function that satisfies
  \begin{equation}\label{eq:Lambertfunction}
    z= W(z)e^{W(z)}
  \end{equation}
  for any complex number $z$.
\end{definition}
We have the following asymptotic approximation of the Lambert $W$ function:
\begin{equation}\label{eq:LambertAsymptotic}
  W(n) = \ln n - \ln \ln n + \frac{\ln \ln n}{\ln n} + \mathcal{O}\left(\frac{(\ln \ln n)^2}{(\ln n)^2}\right).
\end{equation}
For properties described in (\ref{eq:Lambertfunction}), (\ref{eq:LambertAsymptotic} ), and other properties of Lambert's $W$ function, see, for example, \cite{CG},\cite{VJ} \cite{IM}.
\begin{theorem}
The asymptotic formula for the  sequence  \( z_n \)  of  positive root of the sequence of  equations
(\ref{eq:original}) is:
\begin{equation}\label{eq:asymptoticsforpositive}
  z_n \sim \frac{n}{\ln n - \ln \ln n}  \quad \text{as } n \to \infty.
\end{equation}
\end{theorem}
\begin{proof}
Applying logarithm to both sides of the equation (\ref{eq:original}) we get

\begin{align*}
  (n+1)\ln (z)  = n \ln (1+ z)  & = n \left[\ln (z) + \ln \left(1+\frac{1}{z}\right)\right] \\
   & =  n \ln (z) + n \left[\frac{1}{z} -\frac{1}{z^2}+\frac{1}{z^3}+...\right] \\
   & \approx n \ln (z) + \frac{n}{z}
\end{align*}
Therefore we have an approximation
\begin{equation}\label{eq:Lambertanalogue}
   \ln z    \approx  \frac{n}{z}, \quad   \text{ or }\quad  z \ln z \approx n .
\end{equation}
Observing the analogy of equation (\ref{eq:Lambertanalogue}) the Lambert's W function (\ref{eq:Lambertfunction}), and considering the first two terms in the asymptotic series (\ref{eq:LambertAsymptotic},) we get

\begin{equation}\label{eq:asymptoticapproxofpositive}
 z_n  =  \frac{n}{W(n)}\sim \frac{ n}{\ln n - \ln \ln n }.
\end{equation}
This proves the theorem.
\end{proof}



\begin{remark}
Expanding the expression in  (\ref{eq:asymptoticapproxofpositive}), we get
\begin{align*}
   z_n \sim \frac{n}{\ln (n)- \ln( \ln (n) }) & =  \frac{n}{\ln(n) \left(1- \frac{\ln (\ln (n))}{\ln (n) }\right)} \\
   &  = \frac{n}{\ln (n)} \left(  1+ \frac{\ln( \ln( n) ) }{\ln (n) }  +  \left( \frac{\ln (\ln (n)) }{\ln (n) }  \right)^2  +... \right)\\
   &\sim  \frac{n}{\ln (n) }+ \frac{n \ln (\ln( n)) }{(\ln (n))^2 }+ \mathcal{O} \left(\frac{ n (\ln (\ln (n) )^2 }{(\ln n)^3}  \right)
\end{align*}
\end{remark}

In Table 1, the sequence of the positive roots of the equation (\ref{eq:original}) are tabulated against the asymptotic approximation that we derived and the absolute deviation of the positive roots from the symptotics sequence.

\begin{table}[H]
\centering
\caption{ The sequence of positive roots  $z_n$  of the equation  \( z^{n+1} = (1 + z)^n \) vs. its asymptotic approximation $ f(n) = \frac{n}{\ln(n) - \ln(\ln(n))} .$ Roots computed numerically with high-precision Newton-Raphson (tolerance \( 10^{-10} \)).}
\label{tab:roots}

\begin{tabular}{@{} C{0.8cm} C{2cm} C{2cm} C{2cm} @{}}
\toprule
\head{n} & \head{Root \( z \)} & \head{\( f(n) \)} & \head{Absolute Deviation} \\
\midrule
5    & 3.5063 & 4.4107 & 0.9044 \\
10   & 5.4263 & 6.8092 & 1.3829 \\
25   & 10.2845 & 12.1951 & 1.9106 \\
50   & 19.3945 & 19.6300 & 0.2355 \\
75   & 26.0249 & 26.2700 & 0.2451 \\
100  & 32.1999 & 32.4800 & 0.2801 \\
200  & 61.9200 & 62.0296 & 0.1096 \\
500  & 145.210 & 145.243 & 0.033 \\
1000 & 282.300 & 282.318 & 0.018 \\
\bottomrule
\end{tabular}

\vspace{0.2cm}
\footnotesize
\textbf{Notes:} \\
1. High-precision Newton-Raphson with initial guess \( z_0 = f(n) \), tolerance \( 10^{-10} \). \\
2. Absolute deviation: \( |z - f(n)| \). \\
3. Asymptotic approximation \( f(n) \) improves with larger \( n \), reflected in decreasing deviations.

\end{table}

\subsection{Negative roots of the equation}

\begin{theorem}
 For each odd $n\in \mathbb{N} $, the equation (\ref{eq:original}) has exactly one negative root in the interval $(-1,0) $.
\end{theorem}
\begin{proof}
  Consider the equation (\ref{eq:original}). Let $x =-y$ where $y> 0 $ be a negative root. Then
  $$ -y = \left(\frac{1-y}{-y}\right)^n = - \left(\frac{1-y}{y}\right)^n . $$
  This implies that for every negative root $x$ of  (\ref{eq:original}) and transformation $y =-x$
  \begin{equation}\label{eq:negativescaling}
   y = \left(\frac{1-y}{y}\right)^n.
  \end{equation}
  Define
  $$    f(y)= \left(\frac{1-y}{y}\right)^n-y. $$
  We analyse $f(y)$ on the interval $(0,1]$.
  $$ \text{ As  }  y \rightarrow 0+,\, f(y) \rightarrow   \infty .$$
  At $y=1, $ we have $y(1)=-1 $. Consequently, by intermediate value theorem, we have at least one value of $y$ for which $f(y)= 0 $ in $(0,1)$. This implies $x=-y \in (-1,0) $ is a solution of the equation (\ref{eq:original}).
  $$ f'(y)= - \frac{n}{y^2} \left(\frac{1-y}{y}\right)^{n-1}- 1.$$
  Since $f'(y) < 0 $,  $f$ is strictly decreasing for all $y \in (0,1) $. This assures the uniqueness of   a negative root for each  odd $n$. We  proceed to prove that the sequence of  negative roots is strictly increasing. Let $m > n $ are any two odd positive integers.
   $$ \text{ Assume that }     1 > y_m \geq y_n  > 0.  \text{ Then }, \frac{1-y_m}{y_m} \leq \frac{1-y_n}{y_n} . $$
  Given that $\frac{m}{n} > 0 $, and $y_n \in (0,1)$,
  $$   y_m=\left( \frac{1-y_m}{x_m}\right )^m \leq   \left( \frac{1-y_n}{y_n}\right )^m = y_n^{\frac{m}{n}} < y_n
  $$
  This leads to a contradiction to the assumption $ y_m \geq y_n $. Therefore. we have
  $$ y_m  \leq y_n \Rightarrow -y_m \geq -y_n \Rightarrow x_m  \geq x_n $$
  Thus, each equation (\ref{eq:original}) for each odd $n \in \mathbb{N}  $ has exactly one negative real solution $x_n $, and these solutions form an increasing sequence.
  \end{proof}

 We proved that the negative roots of the equation (\ref{eq:original}) are accumulated in the interval $(-1,0) $ . The sequence of negative numbers form an increasing sequence  and bounded above by $0$  and  hence a convergent  and  converges to its least upper bound. In the next theorem we find the limit to which the sequence of negative roots of the (\ref{eq:original} converges.
 \begin{theorem}
  Let $\{x_n\}$, where $n\in \mathbb{N} $ is odd,  is the sequence  negative roots of the  polynomial equation (\ref{eq:original}). Then the sequence   $x_n \rightarrow -\frac{1}{2} $ as $n\rightarrow \infty $
\end{theorem}
 \begin{proof}
   Let $n\in \mathbb{N} $ is odd. Let $x_n $ is the negative solution of the equation(\ref{eq:original}). Let $x_n= -y_n $. Then
   $$ y_n= \left( \frac{1}{y_n}-1\right)^n . $$
   Taking the logarithms both sides we get
   \begin{equation}\label{eq:logofnegativeroot}
     \ln (y_n)= n \ln \left(\frac{1}{y_n}-1 \right).
   \end{equation}
   This  since $ n  $ can be arbitrarily large the equality in can be true if the expression on the right hand side is bounded. That is:
    \begin{align*}
     \Leftrightarrow & \ln \left(\frac{1}{y_n}-1\right) \rightarrow  0 \\
      \Leftrightarrow & \left(\frac{1}{y_n}-1\right) \rightarrow  1 \\
      \Leftrightarrow & y_n \rightarrow \frac{1}{2} \Leftrightarrow  x_n \rightarrow -\frac{1}{2}
     \end{align*}
 \end{proof}

  \begin{remark}
    Since the  negative roots of the equation  form strictly increasing sequences, with the knowledge of $x_1=-61803... = -\frac{1}{\varphi}$, where $\varphi$ is the golden ratio, we conclude that all the negative roots of (\ref{eq:original}) are accumulated in the interval $[-\frac{1}{\varphi}, -\frac{1}{2})$.
  \end{remark}

\begin{theorem}
  The equation (\ref{eq:original}) has no  negative real solution if $n$ is even.
\end{theorem}

\begin{proof}
   Suppose that $ \bar{x}$ is a any real root of the equation (\ref{eq:original}) for some even $n  \in \mathbb{N }$. Then,
$$  \bar{x} = \left(1+\frac{1}{\bar{x}}\right)^n    \geq  0 $$
which is a contradiction with the fact that $\bar{x} < 0 $. This proves the Theorem.
\end{proof}

\subsection{ Asymptotic approximation of the sequence of negative roots of the equation}
\begin{theorem}
  The sequence $x_n$ of negative real roots of the equation (\ref{eq:original}) converge to $-0.5$ with an asymptotic approximation
  \begin{equation}\label{eq:asymptoticapproxfornegative}
    x_n \sim -0.5 - \frac{\ln n }{4n}.
  \end{equation}
\end{theorem}
\begin{proof}
  Let us write
  \begin{equation}\label{eq:approximatetrialsolution}
    x_n = -0.5 + \epsilon_n,
  \end{equation}
 where $  \epsilon_n $ is a corrective term,  and $  \epsilon_n \rightarrow 0 $ as $ n \rightarrow \infty $.
 Replacing the equation (\ref{eq:approximatetrialsolution}) into the equation (\ref{eq:original}), we get
 \begin{equation}\label{eq:negativeapproximateroot}
   ( -0.5 + \epsilon_n )^{n+1} =  ( 0.5 + \epsilon_n )^{n} .
 \end{equation}
  Let us use  the approximation
  $$ \ln (1  \pm 2 \epsilon_n) \approx  \pm 2 \epsilon_n. $$
   Let us Apply  logarithmic both side of the equation (\ref{eq:negativeapproximateroot}). For the left hand side of the equation(\ref{eq:negativeapproximateroot}) we get
  \begin{align*}
   (n+1)\ln (-0.5 +\epsilon_n) &  = (n+1)\ln(-0.5)+ (n+1)\ln (1- 2\epsilon_n).  \\
    & \approx (n+1)\ln(-0.5)  -2(n+1)\epsilon_n
   \end{align*}
 For the right hand expression in equation (\ref{eq:negativeapproximateroot}) we get
 we get
 \begin{align*}
   n\ln (0.5 +\epsilon_n) &  = n\ln(0.5)+ n\ln (1+ 2\epsilon_n),  \\
    & \approx n\ln(0.5) + 2n \epsilon_n.
   \end{align*}
   With $\ln (-0.5 )= \ln 0.5 + i\pi $, considering only the real parts of the right and the left hand side expressions, we get
   \begin{equation}\label{eq:realpartofepsilon}
      (n+1)\ln(0.5)  -2(n+1)\epsilon_n =  n\ln(0.5) + 2n \epsilon_n .
   \end{equation}
From (\ref{eq:realpartofepsilon}) we get
    \begin{equation}\label{eq:correctivetermforthenegative}
         \epsilon_n \approx -\frac{\ln 2}{4n} .
    \end{equation}
  Replacing the approximate quantity  $ \epsilon_n$ in (\ref{eq:correctivetermforthenegative}) into (\ref{eq:approximatetrialsolution}) yields the desired result.
\end{proof}
In Table 2, we tabulated the sequence of the negative roots against the asymptotic approximations.

\begin{table}[H]
\centering
\caption{Negative roots of \( z^{n+1} = (1+z)^n \) vs. asymptotic approximation \( f(n) = -0.5 - \frac{\ln(2)}{4n} \). Roots computed numerically with Newton-Raphson (tolerance \( 10^{-6} \)). Negative roots exist for odd \( n \) only.}
\label{tab:negative_roots}

\begin{tabular}{@{} C{0.8cm} C{2cm} C{2cm} C{2cm} @{}}
\toprule
\head{n} & \head{Root \( z \)} & \head{\( f(n) \)} & \head{Absolute Deviation} \\
\midrule
5    & -0.5312 & -0.5347 & 0.0035 \\
15   & -0.5093 & -0.5116 & 0.0023 \\
25   & -0.5051 & -0.5069 & 0.0018 \\
35   & -0.5036 & -0.5049 & 0.0013 \\
45   & -0.5029 & -0.5038 & 0.0009 \\
55   & -0.5023 & -0.5031 & 0.0008 \\
\bottomrule
\end{tabular}

\vspace{0.2cm}
\footnotesize
\textbf{Notes:} \\
1. Roots calculated using Newton-Raphson with \( z_0 = f(n) \). \\
2. Absolute deviation: \( |z - f(n)| \). \\
3. Analytic \( f(n) \) approximates the negative root for large odd \( n \).

\end{table}

We have discussed the positive and the negative real roots of  the equation(\ref{eq:original}). The real roots of  the equation  (\ref{eq:original}) can be considered as the $x$-coordinates of the points of intersections of the two plane curves:
\begin{equation}\label{eq:intersectionoftwocurves}
  y= x^{n+1}, \quad  \text{ and } \quad      y =(1+x)^n ,\quad  n \in \mathbb{N}.
\end{equation}
 If $n$ is odd we have two points of intersection of the two curves one with positive $x$- coordinate and the other with negative $x$- coordinate. If $n$ is even then we have only one point of intersection of the two curves with a positive x-coordinate. As a prototype of this situation look the cases for $n=2$ and $n=3$ in the graph shown below. Note that the coordinates of the  points of intersections are approximate results.


\begin{figure}[H]
    \centering
    \begin{subfigure}{0.48\textwidth}
        \centering
        \begin{tikzpicture}
            \begin{axis}[
                axis lines = middle,
                axis line style={->},
                xlabel = $x$,
                ylabel = $y$,
                xmin = -1.5,
                xmax = 3,
                ymin = -1,
                ymax = 50,
                domain = -1.5:2.7,
                samples = 200,
                legend pos = north west,
                width = \textwidth,
                enlarge x limits=true,
                enlarge y limits=true,
            ]
                \addplot [blue, thick] {x^4};
                \addplot [red, thick] {(1+x)^3};
                \addplot [black, only marks, mark=*, mark size=2pt] coordinates {
                    (-0.550, 0.092)
                    (2.631, 47.82)
                };
                \node at (axis cs: -0.55, 0.1) [anchor=south east, font=\scriptsize] {$(-0.55, 0.09)$};
                \node at (axis cs: 2.63, 47.8) [anchor=south west, font=\scriptsize] {$(2.63, 47.8)$};
                \legend{$y = x^4$, $y = (1+x)^3$}
            \end{axis}
        \end{tikzpicture}
        \caption{}
    \end{subfigure}
    \hfill
    \begin{subfigure}{0.48\textwidth}
        \centering
        \begin{tikzpicture}
            \begin{axis}[
                axis lines = middle,
                axis line style={->},
                xlabel = $x$,
                ylabel = $y$,
                xmin = -1.5,
                xmax = 3,
                ymin = -1,
                ymax = 15,
                domain = -1.5:2.7,
                samples = 200,
                legend pos = north west,
                width = \textwidth,
                enlarge x limits=true,
                enlarge y limits=true,
            ]
                \addplot [blue, thick] {x^3};
                \addplot [red, thick] {(1+x)^2};
                \addplot [black, only marks, mark=*, mark size=2pt] coordinates {
                    (2.148, 9.915)
                };
                \node at (axis cs: 2.15, 10) [anchor=south west, font=\scriptsize] {$(2.15, 9.92)$};
                \legend{$y = x^3$, $y = (1+x)^2$}
            \end{axis}
        \end{tikzpicture}
        \caption{}
    \end{subfigure}
    \caption{Comparison of functions and their intersections. Left: $y = x^4$ and $y = (1+x)^3$ with intersections at $(-0.55, 0.09)$ and $(2.63, 47.8)$. Right: $y = x^3$ and $y = (1+x)^2$ with intersection at $(2.15, 9.92)$.}
\end{figure}
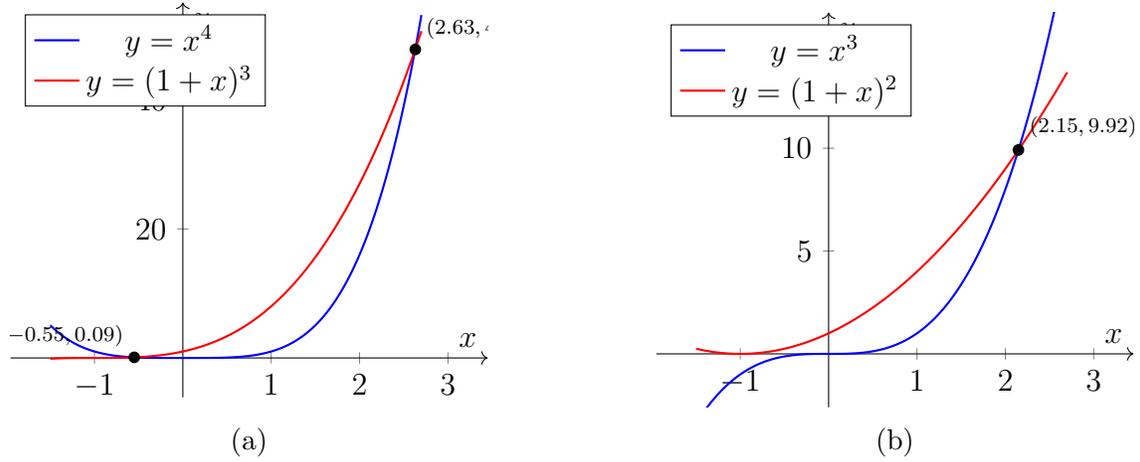

\subsection{Non-real Complex roots of the equation}

\begin{theorem}
   Let $n\in \mathbb{N} $.  The total number $c_n$ of \emph{non-real complex} roots of  the equation (\ref{eq:original}) is
   $$  c_n = n + \frac{(-1)^n-1}{2}= n- (n \text{ mod } 2 ) . $$
The non-real complex roots form $ \frac{c_n}{2}$ conjugate pairs of complex roots.
\end{theorem}

\begin{proof}
   According to the fundamental theorem of algebra, the equation (\ref{eq:original}) has a total of  $n+1$ number of  roots. If $n$ is odd, we have one negative real root, and one positive real root. Subtracting these two roots we have $(n-1)$ non-real complex roots, which form $\frac{n-1}{2}$ conjugate pairs of non-real complex roots.  If $n$ is even, the equation (\ref{eq:original}) has one positive real root, and has no negative real root. Consequently,  the equation (\ref{eq:original})  has $n$ non-real complex roots which form $\frac{n}{2}$ conjugate pairs. Generalizing thetwo cases of odd $n$ and even $n$ we get the desired result.
\end{proof}

\begin{theorem}
 Let $n \in  \mathbb{ N} $. With the exception of the postive real roots, the sequence $z_n$ of the roots of the equation (\ref{eq:original}) asymptotically lie about the vertical line $\Re(z)= -\frac{1}{2}$.
\end{theorem}

\begin{proof}
  Let $z$ be a complex root of the equation (\ref{eq:original}). Taking the modulus of booths sides in  the equation (\ref{eq:original}) we get.

  \begin{equation}\label{eq:Modulusequation}
  |z|^{n+1} = |z+1|^n
  \end{equation}
  From (\ref{eq:Modulusequation}) we get $|z| = |z+1|^{\frac{n}{n+1}}$. As $n \rightarrow \infty $, we get  $ |z+1|^{\frac{n}{n+1}}\rightarrow |z+1|  $,  that yields the equation $|z|=|z+1|. $ This final equation is satisfied by the set of all points  on the vertical line $ \Re(z)= -\frac{1}{2}$, which is the perpendicular bisector of the  segment joining the points $(-1,0)$ and $(0,0)$. Hence the complex roots asymptotically accumulate closer to the vertical line $ \Re(z) = - \frac{1}{2} $.
\end{proof}

\begin{theorem}
 Let $z=x+iy $ is a non-real complex root of the equation (\ref{eq:original}) for some $n \in \mathbb{N } $. Then  the real part $x = \Re (z) = - \frac{1}{2}$,  if and only if   $ z = \frac{-1\pm \sqrt{3} i }{2}$.
\end{theorem}

\begin{proof}
  Put $x= -\frac{1}{2}$. Then $|z|=  |z+1| = \sqrt{\frac{1}{4} + y^2} $. Therefore,  from the equation (\ref{eq:original}),  we get
  $$ \left(\sqrt{\frac{1}{4} + y^2}\right)^{n+1}=  \left(\sqrt{\frac{1}{4} + y^2}\right)^n . $$
  It follows that $   \sqrt{\frac{1}{4} + y^2} =1 $, so that $ y = \pm \frac{\sqrt{3}}{2}$. Obviously  for $ z = \frac{-1\pm \sqrt{3} i }{2}$  the real parts of $ \Re (z) = -1/2 $.
\end{proof}

\begin{theorem}
 Let $z = x+iy $ is a root of the equation (\ref{eq:original}) for some $n \in \mathbb{N }$. Then $ |z|=1 $  if and only if $ z = \frac{-1\pm \sqrt{3}i}{2}$.
\end{theorem}

\begin{proof}
  Let  $z$ is a root of the equation (\ref{eq:original}) for some $n \in \mathbb{N }$. If  $|z|=1$, then
  \begin{align*}
    1= |z|^{n+1} = |z+1|^n & \Rightarrow |z|= |z+1| \\
     & \Rightarrow |z|^2 = |z+1|^2 \\
     & \Rightarrow z \bar{z} = z \bar{z} + z+ \bar{z} + 1 \\
     & \Rightarrow \Re (z)= -\frac{1}{2}
  \end{align*}
  Therefore by the previous Theorem it follows that $ z = \frac{-1\pm \sqrt{3} i}{2}$.
\end{proof}
Also we notice that $ z = \frac{-1\pm \sqrt{3} i}{2}$ are roots for the equation (\ref{eq:original}) when $n = 4 $. Next we show that for each $ n \in \mathbb{N} $,  the multiplicity of  each root $z$ the equation  (\ref{eq:original}) is of multiplicity equal to $1$.
\begin{theorem}
  Each of the roots of the equation (\ref{eq:original}) is of multiplicity equal to $1$.
\end{theorem}

\begin{proof}
  Define $f(x) =x^{n+1}-(1+x)^n $. Then $f'(x)= (n+1)x^n -n (1+x)^{n-1}$. Let us  assume that $\alpha $ satisfies $f(\alpha)=f'(\alpha)= 0 $. That is,
  \begin{equation}\label{eq:fequationofalpha}
    \alpha^{n+1}=(1+\alpha)^n,
    \end{equation}
  and
    \begin{equation}\label{eq:fprimeequationofalpha}
      (n+1)\alpha^{n} = n(1+\alpha)^{n-1} .
    \end{equation}
    Substituting equation (\ref{eq:fequationofalpha}) into equation (\ref{eq:fprimeequationofalpha}) we get
    $$ (n+1) \alpha ^n = n \frac{\alpha^{n+1}}{\alpha +1 }  . $$
    Simplifying yields,
    $$   \alpha^n (n+1+ \alpha) = 0. $$
    However $\alpha =0 $ is not a root, as $\alpha =-n-1 $ is not. There is no common root of $f(x)$ and $f'(x)$ and the assumption  is  wrong. The  the greatest common factor of $f(x)$ and $f'(x)$ is equal to $1$ and every root is simple (of multiplicity equal to $1$).
\end{proof}

We have seen that  every root of the equation (\ref{eq:original}) is simple. Next we discover common roots of  the equation  for multiple values of $n$.
 \begin{theorem}
   Let $\alpha$ is a common root of the equations $z^{n+1}= (1+z)^n $ and $z^{m+1}= (1+z)^m $, for distinct $ m, n \in \mathbb{N} $. Then $\frac{1 \pm \sqrt{3}i }{2}$.
 \end{theorem}

 \begin{proof}
   From the  conditions $\alpha^{n+1}= (1+ \alpha)^n $ and $\alpha^{m+1}= (1 + \alpha)^m $ we get $ \alpha^{n-m}= (1+ \alpha)^{n-m}$. Taking the modulus of this last equation yields $|\alpha|= |\alpha +1 | $.  This yields $ \alpha = \frac{1 \pm \sqrt{3}i }{2}$.
 \end{proof}

 We have proved that the possible common roots of the equation (\ref{eq:original}) for multiple values of $n \in \mathbb{N }$ are $\frac{1 \pm \sqrt{3}i }{2}$. In the next example we show that $\frac{1 \pm \sqrt{3}i }{2}$ are roots of the  equation (\ref{eq:original}) for $n=4$.

 \begin{example}
  For $n=4$ the complex conjugate pair of roots $-\frac{1}{2} \pm \frac{\sqrt{3}}{2} i    $ are roots with moduli equal to $1$. If $z=  -\frac{1}{2} + \frac{\sqrt{3}}{2} i =   e^{i\frac{2\pi}{3}}$ then $z^5=e^{i\frac{4\pi}{3}} $, and $(1+z)^5 = e^{i\frac{4\pi}{3}}$.
\end{example}

In the next theorem, we state  the set of all positive integers $n$ for which the numbers $ \frac{1 \pm \sqrt{3}i }{2}$ are common roots.
\begin{theorem}
   The set  of all $n \in \mathbb{N} $ for which $z= e^{\frac{2 \pi i }{3}},e^{-\frac{2 \pi i }{3}}$  are roots of the equation (\ref{eq:original}) form the sequence $n=4,10,16,... (n \equiv 4 \text{ mod } 6 ) $.
\end{theorem}
\begin{proof}
  Let $z= e^{\frac{2 \pi i }{3}} $. Then $ 1+ z = e^{\frac{ \pi i }{3}} $. So
  $$  z^5 = (e^{\frac{2 \pi i }{3}})^5= e^{\frac{10 \pi i }{3}}= e^{\frac{4 \pi i }{3}} = (1+z)^4 . $$
  This proves that $z= e^{\frac{2 \pi i }{3}} $ is a root of the equation (\ref{eq:original}) for $n = 4 $. Now for any $n = 6m + 4 $ for some $m \in \mathbb{N} $, we have
  $$z^{n+1}= z^{6m+5} =  z^{6m}z^5 = z^5 = (1+z)^4= (1+z)^{6m+4}= (1+z)^n  $$
  Because $n=6 $ is the smallest positive integer such that $z^6=1= (1+z)^6 $, when $z= e^{\frac{2 \pi i }{3}} $. Since non-real complex roots of the equation (\ref{eq:original}) exist in conjugate pairs, the case of $z= e^{-\frac{2 \pi i }{3}} $ is treated  similarly.  This proves the Theorem.
\end{proof}

\subsection{ The regions of non-existence of non-real complex roots }

\begin{theorem}
Let
\begin{equation}\label{eq:symmetricdifference}
  R_1 = \{z \in \mathbb{C} : (|z|-1)(|z+1|-1) < 0 \}, \quad  \text{ See Figure 1 }.
\end{equation}
 Then the equation (\ref{eq:original}) has no solution in the region $R_1$ of the complex plane.
\end{theorem}

\begin{proof}
  Note that the region $R_1$, as defined in (\ref{eq:symmetricdifference}), is the symmetric difference of the regions interior to  the unit circles $|z|=1$ and $|z+1|=1 $. Let $z$ be any root of (\ref{eq:original}). If $|z|< 1$, then
  $$ 1 > |z| > |z|^{n+1} = |z+1|^n. $$
  Hence, $|z+1| < 1$. If $|z+1|< 1$, then

  $$ 1 > |z+1| \geq |z+1|^{n} = |z|^{n+1}. $$
  Hence, $|z| < 1$.
   Therefore, any root located in the interior of either of the two circles must lie within the  intersection of their interiors. This proves the theorem.
\end{proof}

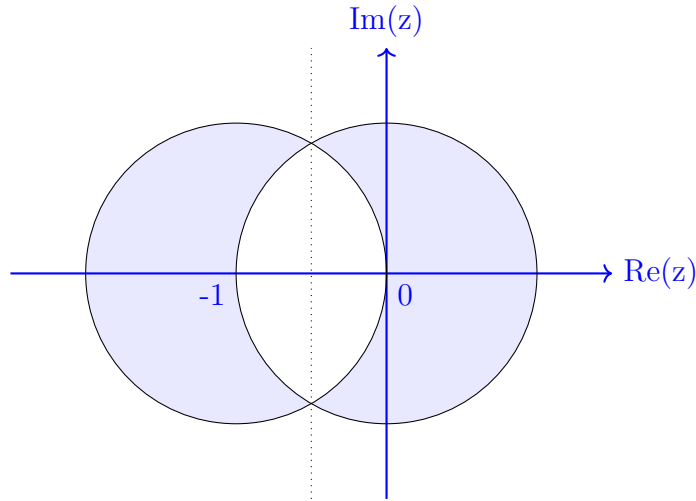
\begin{figure}[htbp]
  \centering
  \begin{tikzpicture}[scale=2]
    \fill[blue!30, opacity=0.3, even odd rule] (0,0) circle (1) (-1,0) circle (1);

    \draw[->, thick, blue] (-2.5,0) -- (1.5,0) node[right] {Re(z)}; 
    \draw[->, thick, blue] (0,-1.5) -- (0,1.5) node[above] {Im(z)}; 

    \draw[dotted] (-0.5,-1.5) -- (-0.5,1.5);

    \draw (0,0) circle (1);
    \draw (-1,0) circle (1);

    \node[blue] at (0,0) [below right] {0};    
    \node[blue] at (-1,0) [below left] {-1};
  \end{tikzpicture}
  \caption{ Shaded region $R_1$.} 
\end{figure}

We have proved that the negative real roots are bounded, contained in the interval $[\frac{1-\sqrt{5}}{2}, -\frac{1}{2}) $ and  that they form an increasing sequence. This interval of the negative real roots  is contained within the  common region of the interiors of the unit circles $|z|=1$ and $|z+1|=1|$. We have demonstrated that there is no roots of  the equation (\ref{eq:original}) within the region of the symmetric difference of the interior regions of the two unit circles.  The remaining questions to be addressed are  wether there exist non-real complex roots  equation (\ref{eq:original}) outside both the circle $|z|=1$ and $|z+1|=1 $, and wether there are regions within the intersection of the interiors of the two unit circles where there are no roots.

\begin{theorem}
  Let $R_2 $ be a region in the complex plane such that
  \begin{equation}\label{eq:region2}
     R_2 = \{ z \in \mathbb{C} : |z+1|> 1, \Re (z) < -\frac{1}{2}\}, \quad \text{ See Figure 2. }
  \end{equation}
Then the equation (\ref{eq:original}) has no solution in the region $R_2$.
\end{theorem}

\begin{proof}
By the definition and the  properties of the region $R_2$ given in (\ref{eq:region2}),

$$ |z|  > |z+1|> 1,  $$
which implies
$$ |z|^{n+1}  >|z|^n > |z+1|^n .  $$
Therefore, there is no solution $z$ of the equation(\ref{eq:original}) in the  region $R_2$.
\end{proof}

\begin{figure}[ht]
\centering
\begin{tikzpicture}

\draw[->] (-3,0) -- (2,0);  
\draw[->] (0,-2) -- (0,2);   

\draw (0,0) circle (1);
\draw (-1,0) circle (1);

\draw[dashed] (-0.5,-2) -- (-0.5,2);

\begin{scope}
  \clip (-3,-2) rectangle (-0.5,2);
  \begin{scope}
    \clip (0,0) circle (1)[reversepath];
    \clip (-1,0) circle (1)[reversepath];
    \fill[blue!25] (-3,-2) rectangle (2,2);
  \end{scope}
\end{scope}

\draw[->,thick] (-3,0) -- (2,0) node[right] {Re(z)};          
\draw[->,thick] (0,-2) -- (0,2) node[above,yshift=1.5pt] {Im(z)}; 

\draw[black] (-3,0) -- (2,0);  
\draw[black] (0,-2) -- (0,2);   

\end{tikzpicture}
\caption{Shaded region $R_2$.}
\end{figure}
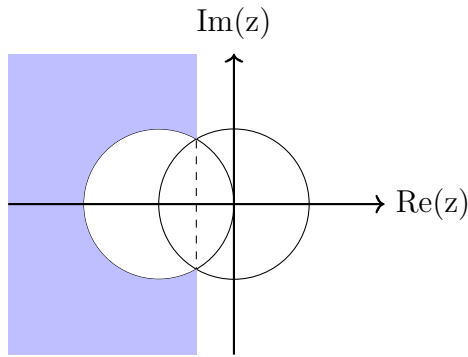

So far in our work,  we have demonstrated that there are no solution to the equation (\ref{eq:original}) within the  region  corresponding to the symmetric difference of the interiors of the unit circles $|z|=1$ and $|z+1| $. Additionally, we have proved that there are no roots of the equation (\ref{eq:original}) outside of the unit circle $|z|=1$ when the condition $\Re (z) < -1/2$ is satisfied.

Now we proceed to investigate wether there are regions  without roots and that lie within the  intersection of the interiors of the unit circles. Our next theorem will establish this condition.

\begin{lemma}[\textbf{Rouch\'{e} Theorem}]\cite{EM}
  Let \( f(z) \) and \( g(z) \) be regular analytic functions of a complex variable \( z \) in a domain \( D \). Let \( \Gamma \) be a simple closed piecewise-smooth curve together with the domain \( G \) bounded by it, which belong to \( D \). If the inequality \( |f(z)| > |g(z)| \) is valid everywhere on \( \Gamma \), then in the domain \( G \), the sum \( f(z) + g(z) \) has the same number of zeros as \( f(z) \).
\end{lemma}

\begin{theorem}
Let $R_3$ be a region in the complex plane:
\begin{equation}\label{eq:region3}
  R_3 =  \left \{z \in \mathbb{C} :  |z+1|< 1, \Re (z) > -\frac{1}{2} \right \}, \text{ See Figure 3 (a)}.
\end{equation}
There is no root of the equation in the region (\ref{eq:original}) in the region of the complex plane $R_3$.
\end{theorem}

\begin{proof}
Let $\epsilon, \delta $ are arbitrary small positive numbest. Consider the following  closed curve $C$ in the complex plane, see Figure 3 (b). For for very small $\epsilon > 0 $, let the segment $AB$ satisfy
$$ z =-\frac{1}{2}+ yi, \quad -\frac{\sqrt{3}}{2} + \epsilon  \leq  y\leq   \frac{\sqrt{3}}{2} - \epsilon   $$
On  segment $AB$,  we have $|z+1|=|z|  < 1 $. Consequently,
$$  |z+1|^n =|z|^n > |z|^{n+1} . $$
Let $Q $ be the point $z=e ^  {i \theta }$, where $ \theta = \delta -\frac{2 \pi}{3} $, and
$P$ is a point $z=e ^  {i \theta }$ where $ \theta = \frac{2 \pi}{3}- \delta $. The curved portion of  $C$ between the points $Q$ and $P$ is an arc of the unit circle $|z|=1 $ parameterized by
$$ z= e^{i \theta}, \quad  -\frac{2 \pi }{3} +\delta \leq \theta \leq \frac{2 \pi }{3} - \delta      $$
$AQ$ and $BP$ are segments of straight lines. On the portion of the closed curve, that satisfy the conditions  $  \Re (z)  >  -\frac{1}{2}, |z|< 1 $, we have $|z+1|> 1 \geq |z|$. Consequently,
$$ |z|^{n+1}  \leq   |z|^n < |z+1|^n  . $$
Overall, $|z|^{n+1}  < |z+1|^n   $ for all $z$ on the closed curve $C$. According to Rouch\'{e}'s Theorem, the functions $(z+1)^n $ and $z^{n+1}-(1+z)^n $ have the same number of zeros inside the closed curve $C$. However, for $(z+1)^n$  all its root $z=-1$ (multiplicity equal to $n$)   are outside the closed curve $ C $. The equation (\ref{eq:original}) has no root inside the closed curve $C$. Since $\epsilon, \delta $ arbitrary small positive numbers, letting  $\epsilon, \delta \rightarrow 0 + $, we get the boundary of the region $R_3$.  The limiting  curve $\delta = \epsilon =0 $, yields the boundary of the region $|z|< 1, \Re (z) >  $. The  points $z= e^{\pm \frac{2 \pi}{3}}$ lie on this boundary of this limiting  curve, are solution to the equation (\ref{eq:original}) for some values of $n$. Yet there are no  roots interior to the limiting curve, that is the region $R_3$. This proves the Theorem.
\end{proof}

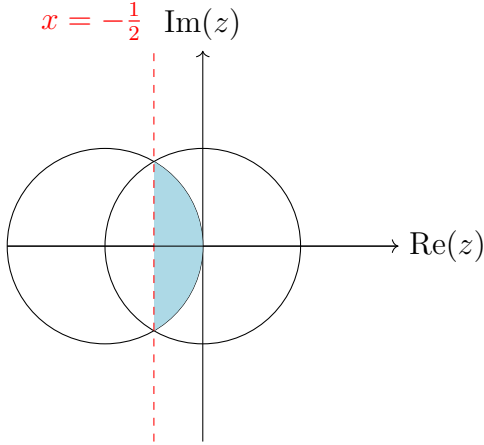
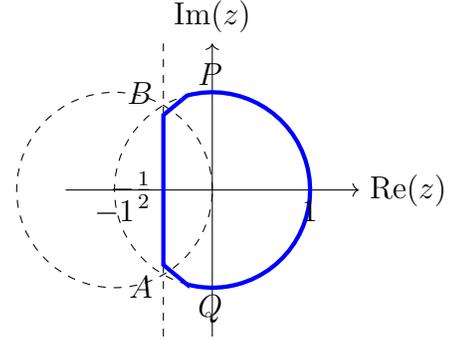
\begin{figure}[H]
\centering
\begin{subfigure}[b]{0.45\textwidth}
\centering
\begin{tikzpicture}[scale=1.3]
\draw[->] (-2,0) -- (2,0) node[right] {$\text{Re}(z)$};
\draw[->] (0,-2) -- (0,2) node[above] {$\text{Im}(z)$};

\draw (0,0) circle (1);
\draw (-1,0) circle (1);

\begin{scope}
  \clip (-1,0) circle (1);
  \clip (-0.5,-2) rectangle (2,2);
  \fill[lightblue] (-0.5,-2) rectangle (2,2);
\end{scope}

\draw[dashed, red] (-0.5,-2) -- (-0.5,2) node[above left] {$x = -\frac{1}{2}$};

\draw (-2,0) -- (2,0);
\end{tikzpicture}
\subcaption{The common interior region of the two unit circle with no roots.}
\label{fig:region1}
\end{subfigure}
\hfill
\begin{subfigure}[b]{0.45\textwidth}
\centering
\begin{tikzpicture}[scale=1.3]
\draw[->] (-1.5,0) -- (1.5,0) node[right] {$\text{Re}(z)$};
\draw[->] (0,-1.5) -- (0,1.5) node[above] {$\text{Im}(z)$};

\draw[dashed] (0,0) circle[radius=1];
\draw[dashed] (-1,0) circle[radius=1];

\draw[dashed] (-0.5,-1.5) -- (-0.5,1.5);

\coordinate (A) at (-0.5, {-sqrt(3)/2 + 0.1});
\coordinate (B) at (-0.5, {sqrt(3)/2 - 0.1});
\coordinate (P) at (105:1);  
\coordinate (Q) at (-105:1); 

\draw[ultra thick, blue]
    (A) --
    (B) --
    (P) arc[start angle=105, end angle=-105, radius=1] --
    (Q) --
    cycle;

\node[below left] at (A) {$A$};
\node[above left] at (B) {$B$};
\node[above right] at (P) {$P$};
\node[below right] at (Q) {$Q$};

\node[below] at (1,0) {$1$};
\node[below] at (-1,0) {$-1$};
\node[left] at (-0.5,0) {$-\frac{1}{2}$};
\end{tikzpicture}
\subcaption{  A curve $C$  used for application of Rouch \'{e}'s Theorem}
\label{fig:region2}
\end{subfigure}
\caption{Zero-free  common interior region for the two unit circles and suitable curve to apply Rouch\'{e}'s theorem}
\label{fig:combined}
\end{figure}


We have proved that,  for each $n \in \mathbb{N} $, the equation (\ref{eq:original}) has a positive real root. In this section, we establish that there is no non-real complex solution  that satisfies the condition $\Re (z) > -\frac{1}{2}$, other than the positive real root.
\begin{theorem}
Let us denote by $R_5 $ a region in the Complex plane :
\begin{equation}\label{eq:region5}
  R_4 = \left \{  z \in \mathbb{C} : |z| > 1, \Re (z) > -\frac{1}{2}  \right \}, \quad \text{See Figure 5.}
\end{equation}
The equation (\ref{eq:original}) has no non-real complex solution in the region $R_4 $ of the complex plane.
\end{theorem}

\begin{proof}
  We assume that there exists a root $z \in R_4 $ and aim to  reach  a contradiction, thereby  affirming the non-existence of such a root. In the region $R_4$, we have the inequality
  \begin{equation}\label{eq:norminequality}
    |z+1| > |z|> 1
  \end{equation}
  Suppose that $z\in R_4 $ is a root of the equation(\ref{eq:original}). Let $\text{Arg}(z):= \theta $, and $\text{Arg}(z+1):= \phi $. In the region $R_4$ we have
  \begin{equation}\label{eq:phithetainequality}
    0 < \frac{\theta}{2} < \phi < \theta < \frac{2 \pi }{3}
  \end{equation}
  \begin{equation}\label{eq:ththaminusphi}
    0 < \theta - \phi <  \frac{\theta}{2} < \phi < \theta < \frac{2 \pi }{3}
  \end{equation}
  By inequality condition (\ref{eq:norminequality}), and the law of cosines, see Figure 6,
  \begin{align*}
    1 & =  |z|^2 +|z+1|^2 -2|z+1||z| \cos (\theta-\phi) \\
     & = |z+1|^2 \left( 1+ \frac{|z|^2}{|z+1|^2}-  2 \cos (\theta-\phi)\frac{|z|}{|z+1|} \right) \\
    & = |z+1|^n \left( 1+ \frac{|z|^2}{|z+1|^2}-  2 \cos (\theta-\phi)\frac{|z|}{|z+1|} \right)^{\frac{n}{2}} \\
    & = |z|^{n+1} \left( \left( \frac{|z|}{|z+1|}- \cos(\theta -\phi)   \right)^2 - \sin ^2 (\theta - \phi) \right)^{\frac{n}{2}}   \\
     & \geq |z|^{n+1}  \left( \frac{|z|}{|z+1|}- \cos(\theta -\phi)   \right)^{n}   \\
     & = |z|^{n+1}  \left( \frac{|z|}{|z+1|}- \frac{(1-|z|^2-|z+1|^2)}{2|z||z+1|}   \right)^{n}   \\
    & = |z|^{n+1}  \left( \frac{|z|}{|z+1|}- \frac{1}{2|z||z+1|} +\frac{|z|}{2|z+1|}+ \frac{|z+1|}{2|z|} \right)^{n}   \\
    & \geq |z|^{n+1}  \left(  1 +\frac{|z|}{|z+1|}- \frac{1}{2|z||z+1|}  \right)^{n}   \\
    & > 1,
  \end{align*}
   Note that the expression $\frac{|z|}{|z+1|}- \frac{1}{2|z||z+1|}  > 0 $ holds by the inequality condition (\ref{eq:norminequality}). This leads to a contradiction, implying t $1 > 1 $. Therefore, we conclude that there is no non-real complex solution of the equation (\ref{eq:original}) in the region $R_4$ of the complex plane.
\end{proof}

%
%
%
%
%
%
%

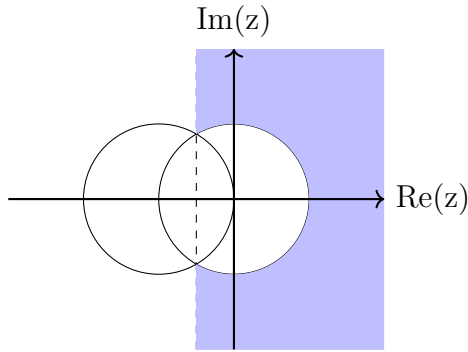
\begin{figure}[H]
\centering
\begin{tikzpicture}

\draw (0,0) circle (1);
\draw (-1,0) circle (1);

\draw[dashed] (-0.5,-2) -- (-0.5,2);

\begin{scope}
  \clip (-0.5,-2) rectangle (2,2);  
  \begin{scope}
    \clip (0,0) circle (1)[reversepath];
    \clip (-1,0) circle (1)[reversepath];
    \fill[blue!25] (-3,-2) rectangle (2,2);
  \end{scope}
\end{scope}

\draw[->,thick] (-3,0) -- (2,0) node[right] {Re(z)}; 
\draw[->,thick] (0,-2) -- (0,2) node[above] {Im(z)};

\end{tikzpicture}
\caption{ Shaded region $R_4$ }
\end{figure}

\begin{figure}[ht]
\centering
\begin{tikzpicture}[scale=1.5, >=Stealth]
    \draw[->] (-1.5,0) -- (3,0) node[right] {$\mathrm{Re}$};
    \draw[->] (0,-0.5) -- (0,3.5) node[above] {$\mathrm{Im}$};

    \coordinate (A) at (-1,0);
    \coordinate (B) at (0,0);
    \coordinate (Z) at (2,2.5);  
    \coordinate (Xaxis) at (1,0);

    \draw[thick] (A) -- (B) -- (Z) -- cycle;
    \draw[dashed] (B) -- (Xaxis);

    \pic[draw=black, angle radius=6mm, angle eccentricity=1.2,
         pic text=$\phi$, pic text options={font=\scriptsize}] {angle = B--A--Z};

    \pic[draw=black, angle radius=8mm, angle eccentricity=1.25,
         pic text=$\theta$, pic text options={font=\scriptsize}] {angle = Xaxis--B--Z};

    \node[sloped, left=2pt] at ($(A)!0.4!(Z) + (-0.1,0.1)$) {$|z+1|$};

    \node[sloped, right=2pt] at ($(B)!0.5!(Z) + (0.1,-0.1)$) {$|z|$};

    \node[below] at (A) {$-1$};
    \node[below left] at (B) {$0$};
    \node[above right] at (Z) {$z$};

    \fill[fill=black] (A) circle (1.2pt);
    \fill[fill=black] (B) circle (1.2pt);
    \fill[fill=black] (Z) circle (1.2pt);
\end{tikzpicture}
\caption{Triangle in complex plane with modified label positioning. Angles measured at vertices with $|z|$ and $|z+1|$ placed on respective sides of segments.}
\end{figure}
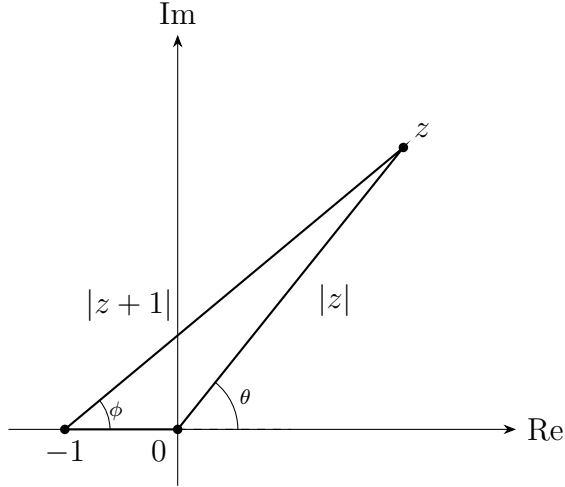

\subsection{ The region of accumulation of all negative and non-real complex roots of  the equation}

In the previous subsection we  explored the regions  with either no roots or with no non-real complex roots. In the subsection, we will highlight the region of accumulation of all roots of (\ref{eq:original}) excluding the positive real root.


\begin{theorem}
    All  the negative  roots (exist for odd $n$ only) as well as the non-real complex roots of the equation (\ref{eq:original}) are accumulated inside and on the boundary of plano-convex shaped region $R_5$, see Figure 7, bounded and the arc of the unit circle,
    $$    z=e^{i\theta}, \quad \frac{2 \pi}{3} \leq  \theta   \leq \frac{4 \pi}{3}, $$
    and by the vertical line segment:
    $$  z = -\frac{1}{2} + iy, \quad -\frac{\sqrt{3}}{2} \leq y \leq \frac{\sqrt{3}}{2} . $$
The only  roots located on the boundary of this region are $\frac{-1 \pm \sqrt{3}}{2}$ for  $n$ values satisfying $n \equiv 4 \text{ mod } 6 $. See fig. 6.
\end{theorem}


  \begin{proof}
    Consider the bilinear transformation (\ref{eq:bilineartransformation}), that yields equation (\ref{eq:transformedequation}). Now suppose that there is a non-real complex root of the equation(\ref{eq:original})
  \end{proof}

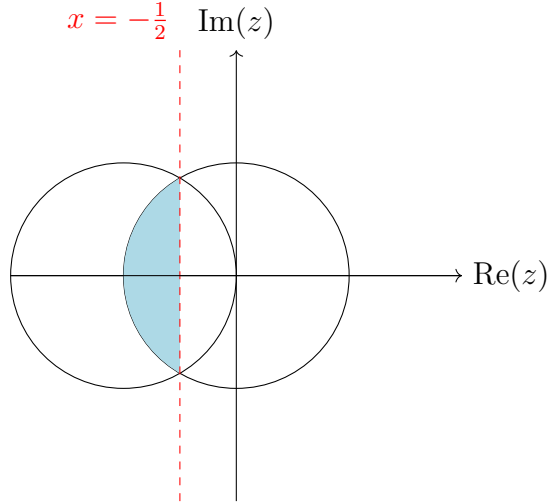
\begin{figure}[H]
\centering
\begin{tikzpicture}[scale=1.5]
\draw[->] (-2,0) -- (2,0) node[right] {$\text{Re}(z)$};
\draw[->] (0,-2) -- (0,2) node[above] {$\text{Im}(z)$};

\draw (0,0) circle (1);
\draw (-1,0) circle (1);

\begin{scope}
  \clip (0,0) circle (1);
  \clip (-2,-2) rectangle (-0.5,2);
  \fill[lightblue] (-2,-2) rectangle (-0.5,2);
\end{scope}

\draw[dashed, red] (-0.5,-2) -- (-0.5,2) node[above left] {$x = -\frac{1}{2}$};

\draw (-2,0) -- (2,0); 

\end{tikzpicture}
\caption{ Shaded region $R_5$.}
\end{figure}

\begin{table}[H]
\centering
\caption{Root Behavior for Odd vs. Even \( n \)}
\begin{tabular}{@{}lll@{}}
\toprule
Property & Odd \( n \) & Even \( n \) \\
\midrule
Negative Real Roots & 1 (Increasing) & None \\
Positive Real Roots & 1   &  1  \\
Non-real Complex Roots & Bounded near \( -\frac{1}{2} \pm i\frac{\sqrt{3}}{2} \) & Bounded near same \\
\bottomrule
\end{tabular}
\end{table}

\subsection{ Asymptotic approximation of the  sequence of  non-real complex roots Equation}

We have seen that the equation (\ref{eq:original}) has  non-real complex roots except for $n=1$, where both roots are pure real roots. For equation (\ref{eq:original}), non-real complex roots always exist in conjugate pairs. In this subsection we see that the non-real complex roots  fall intro families of roots. The family of roots that converge to the complex roost $-\frac{1}{2} \pm \frac{\sqrt{3}}{2}i $.

\begin{theorem}
  The non-real complex roots near the unit circle for large $n$ are asymptotically approximated by:
  \begin{equation}\label{eq:nearcirclerootsapproximation}
    z_n=  e^{\frac{2 \pi }{3}} \left(1 + \frac{ \pi \sqrt{3}}{n}   -\frac{ \pi i}{3n} \right) + \mathcal{O}\left(\frac{1}{n^2}\right).
  \end{equation}
\end{theorem}

\begin{proof}

 Let us Carefully analyse the argument  of the limit point $e^{i 2\pi /3 }$ of the sequence of non-real complex roots of the equation (\ref{eq:original}). We analyze the perturbation of the argument  $\frac{2\pi}{3} $.   Let $\theta = \frac{2\pi}{3} + \delta $, where $\delta = \mathcal{O}(1/n) $.

  \begin{equation}\label{eq:znplusone}
   z^{n+1}= e^{i( 2 \pi/3  + \delta) (n+1)} = e^{ 2i (n+1)\pi/3 }  e^{ \delta i (n+1)}= e^{ \frac{2 n \pi i }{3}  } e^{ \frac{2 \pi i}{3} }    e^{i n \delta}  e^{i\delta}
\end{equation}

\begin{align}
  (1+z)^n  &= \left(1 + e^{i\frac{2 \pi}{3}} e^{i \delta}\right)^n   \nonumber \\
    &= \left(1 + e^{i\frac{2 \pi}{3}} (1 + i \delta - \frac{\delta^2}{2} + \ldots)\right)^n    \nonumber \\
   &= \left(e^{i\frac{ \pi}{3}} + e^{i\frac{2 \pi}{3}} i \delta - e^{i\frac{2 \pi}{3}} \frac{\delta^2}{2} + \ldots\right)^n  \nonumber \\
   &= e^{i\frac{ n \pi}{3} } \left(1 + i \delta e^{i\frac{ \pi}{3}} - \frac{\delta^2}{2}e^{i\frac{ \pi}{3}} + \ldots  \right)^n  \label{eq:zplusonen}
\end{align}

Substituting equation (\ref{eq:znplusone}) and (\ref{eq:zplusonen}) into the original equation (\ref{eq:original}) we get:

\begin{equation}\label{eq:perturbedequation}
 e^{ \frac{2 n \pi i }{3}  } e^{ \frac{2 \pi i}{3} }    e^{i n \delta}  e^{i\delta} = e^{i\frac{ n \pi}{3} } \left(1 + i \delta e^{i\frac{ \pi}{3}} - \frac{\delta^2}{2}e^{i\frac{ \pi}{3}} + \ldots  \right)^n
\end{equation}
Taking the logarithms both sides  of the equation (\ref{eq:perturbedequation}), with approximation
$$  \ln (1+x) \approx x,\quad  |x| \ll 1,   $$
we get :

\begin{equation}\label{eq:firstlogequation}
    \frac{2 n\pi i}{3}   + \frac{2 \pi i}{3}  +   i n  \delta  +   i \delta  = \frac{i n \pi}{3}   + n \delta i  e^{i \pi /3} + 2k \pi i, \quad k \in \mathbb{Z}.
\end{equation}

Collecting like terms  in ( \ref{eq:firstlogequation})
\begin{equation}\label{eq:secondlogequation}
    \frac{ n\pi i}{3}   + \frac{2 \pi i}{3}  +   i n  \delta  +   i \delta  =  n \delta i  e^{i \pi /3} + 2k \pi i, \quad k \in \mathbb{Z}.
\end{equation}

 Let us  set
 \begin{equation}\label{eq:settingt}
   k = \frac{n }{6} +t, \quad t\in \mathbb{Z},
 \end{equation}
  to cancel-out the leading $\mathcal{O}(n) $ term. The resulting replacement followed by cancellation yields,

  \begin{equation}\label{eq:firsttequation}
      \frac{2 \pi i}{3}  +   i n  \delta  +   i \delta  =  n \delta i  e^{i \pi /3} + 2 \pi t  i, \quad t \in \mathbb{Z}.
  \end{equation}

  By the assumption that  $ \delta = \mathcal{O} (\frac{1}{n})  $,  the term $ i \delta \rightarrow 0$ as $n \rightarrow \infty $. Omitting this term, and  putting $t=0$ for minimal correction, we get the final equation

  \begin{equation}\label{eq:tremoved}
      \frac{2 \pi }{3}  +    n  \delta    =  n \delta   e^{i \pi /3}.
  \end{equation}

 From (\ref{eq:tremoved}) we obtain
 \begin{equation}\label{eq:deltacalculated}
    \delta = \frac{2 \pi}{3n } e^ {- \frac{i 2 \pi }{3}} =  -\frac{\pi}{3n}(1+\sqrt{3} i).
 \end{equation}

Replacing the calculated  into the initial perturbation we get

\begin{equation}\label{eq:zsubnapprox}
   z_n = e^{\frac{2 \pi i}{3}} e^{i \delta}\approx  e^{\frac{2 \pi i}{3}} (1+ i\delta) = e^{\frac{2 \pi }{3}} \left( 1 + \frac{\sqrt{3} \pi}{3n}-  \frac{\pi i }{3n}\right)
\end{equation}

\begin{align}
 z_n & \approx e^{\frac{2 \pi i}{3}} e^{i \delta}    \nonumber  \\
   &    \approx e^{\frac{2 \pi i}{3}} (1 +i \delta) + \mathcal{O}(\frac{1}{n^2})    \nonumber  \\
   & \approx   e^{\frac{2 \pi i }{3}} \left(1 + \frac{ \pi \sqrt{3}}{n}   -\frac{ \pi i}{3n} \right) + \mathcal{O}\left(\frac{1}{n^2}\right)
\end{align}

\end{proof}

\begin{remark}
 In same method as we have derived for the  asymptotic approximation the roots near $e^{\frac{2 \pi }{3}} $, we can derive the asymptotic series approximation for the roots near  $e^{-\frac{2 \pi }{3}} $, given by
        $$ z_n =  e^{\frac{-2 \pi }{3}} \left(1 - \frac{ \pi \sqrt{3}}{n}   + \frac{ \pi i}{3n} \right) + \mathcal{O}\left(\frac{1}{n^2}\right)  $$
\end{remark}

\begin{remark}
  Instead  of assuming the perturbation of the form that we have used here, if we use a perturbation of the form, $z = e^{-\frac{2 \pi i}{3}}   + \delta   $, where $\delta = \mathcal{O}(1/n)$ we arrive at the same result.
\end{remark}

Next, we establish some propositions that are the byproducts of the current problem and the results we have obtained. Many of these results involve algebraic equations in two variables, which might otherwise be difficult to prove.

\begin{theorem}
  For  any root  $z = x+iy \in \mathbb{C}  $  of  the equation (\ref{eq:original}), the ordered pair $(x,y)\in \mathbb{R}^2 $  is a root of   the simultaneous polynomial equations of degree $n+1$ in variables  $x$ and $y$:

  \begin{equation}\label{eq:decomposedcartesian}
   \begin{cases}
\displaystyle \sum_{k=0}^{\lfloor (n+1)/2 \rfloor} (-1)^k \binom{n+1}{2k} x^{n+1-2k} y^{2k} = \sum_{k=0}^{\lfloor n/2 \rfloor} (-1)^k \binom{n}{2k} (1 + x)^{n-2k} y^{2k} \\
\displaystyle \sum_{k=0}^{\lfloor n/2 \rfloor} (-1)^k \binom{n+1}{2k+1} x^{n-2k} y^{2k+1} = \sum_{k=0}^{\lfloor (n-1)/2 \rfloor} (-1)^k \binom{n}{2k+1} (1 + x)^{n-2k-1} y^{2k+1}
\end{cases}
\end{equation}

\end{theorem}

\begin{proof}

Substituting $z=x+iy$ into the left hand side of the original equation(\ref{eq:original}), and separating the real and imaginary parts we get
\begin{align}\label{eq:leftalgebraicsubstitution}
     (x + iy)^{n+1} & = \sum_{k=0}^{n+1} \binom{n+1}{k} x^{n+1-k} (iy)^k = \sum_{k=0}^{\lfloor (n+1)/2 \rfloor} (-1)^k \binom{n+1}{2k} x^{n+1-2k} y^{2k}  \nonumber \\
       & +  i \sum_{k=0}^{\lfloor n/2 \rfloor} (-1)^k \binom{n+1}{2k+1} x^{n-2k} y^{2k+1}.
    \end{align}
Substituting $z=x+iy$ into the right hand side of the original equation (\ref{eq:original}), and then separating the real and imaginary parts we get

\begin{align}\label{eq:rightalgebraicsubstitution}
     (1 + x + iy)^n & = \sum_{m=0}^n \binom{n}{m} (1 + x)^{n-m} (iy)^m  = \sum_{k=0}^{\lfloor n/2 \rfloor} (-1)^k \binom{n}{2k} (1 + x)^{n-2k} y^{2k}   \nonumber \\
       & +  i \sum_{k=0}^{\lfloor (n-1)/2 \rfloor} (-1)^k \binom{n}{2k+1} (1 + x)^{n-2k-1} y^{2k+1}.
    \end{align}
The result of the Theorem follows by equating the real parts of (\ref{eq:leftalgebraicsubstitution}) with the real part of (\ref{eq:rightalgebraicsubstitution}), and by equating the imaginary part of (\ref{eq:leftalgebraicsubstitution}) with the imaginary part of (\ref{eq:rightalgebraicsubstitution}).
\end{proof}

\begin{theorem}
Let $(x,y) \in \mathbb{R}^2 $ is a solution of the systems of equation (). Then the complex number $x+iy$ is the root of the equation (\ref{eq:original}). Furthermore,
\begin{itemize}
    \item $x$ is a negative real root of the equation (\ref{eq:original}), if $(x,0)$ is a root of the system of equations (\ref{eq:decomposedcartesian}) and $x< 0 $.
    \item $x$ is a positive real root of the  equation (\ref{eq:original}), if $(x,0)$ is a root of the system of equations (\ref{eq:decomposedcartesian}) and $x> 0 $.
    \item $z=x+iy$ is a non-real complex root of the equation (\ref{eq:original}), if $(x,y)$ is a root of the system of equations (\ref{eq:decomposedcartesian}) and $ y \neq 0 $.
  \end{itemize}
 \end{theorem}



 According to B\'{e}zout's theorem (See, for example, \cite{DP}), two polynomial equations of degrees $m$ and $n$ in two variables, without common factors, considered over an algebraically closed field, will have at most $mn$ intersection points when counted with multiplicities. This includes points at infinity and points with complex coordinates. According to theorem the system of equations (\ref{eq:decomposedcartesian} ) has $(n+1)^2$ solutions $(x,y) \in \mathbb{C}^2 $. The results from the fundamental theorem of Calculus, and the results that we obtained here in this paper yield that the system of polynomial equation (\ref{eq:decomposedcartesian} ) has exactly $(n+1)$ roots in $ \mathbb{R}^2 $. The remaining $(n(n+1)$ roots have at least one complex coordinates.

\begin{theorem}
  The system of equation (\ref{eq:decomposedcartesian}) of polynomial equations of degree $(n+1)$ has exactly $(n+1)$ roots  for any $ n \in \mathbb{N} $.
\end{theorem}

\begin{theorem}
  Let $z= r e^{i\theta}, r >0, -\pi < \theta \leq \pi $, be any root of the equation (\ref{eq:original}) in  polar form. Then the ordered pair $ (r, \theta) $ satisfies the system of equations:
\begin{equation}\label{eq:polardecomposition}
  \begin{cases}
     &  r^{n+1} \cos ((n+1)\theta) =   \sum_{m=0}^{n} \binom{n}{m} r^m \cos (m \theta)\\
     & r^{n+1} \sin ((n+1)\theta) =   \sum_{m=1}^{n} \binom{n}{m} r^m \sin (m \theta)
  \end{cases}
\end{equation}
\end{theorem}

Based on the analysis of the nature of roots and their placement in the complex plane, we formulate the following theorem:
\begin{theorem}
  Consider the equation
\begin{equation}\label{eq:realpartofpolardecomposition}
  r^{n+1} \cos ((n+1)\theta) =   \sum_{m=0}^{n} \binom{n}{m} r^m \cos (m \theta),
\end{equation}
where $r > 0 $, and $ 0 \leq  \theta \leq    \pi $. Then
  \begin{itemize}
    \item  for any fixed $ 0 < \theta < \pi $, the equation (\ref{eq:realpartofpolardecomposition}) is not satisfied  for any  $r >1$,
    \item  for $ \theta =0 $, the equation (\ref{eq:realpartofpolardecomposition}) is satisfied  for  a unique value of  $r > 1 $.
\item  for $ \theta = \pi  $, the equation (\ref{eq:realpartofpolardecomposition}) is satisfied  for  a unique value of $ 0 < r < 1$.
  \end{itemize}

\end{theorem}

We have mentioned that the equations (\ref{eq:original}) can be viewed as the generalization of the  well known equation $x^2=1+x $, whose positive root is  the  golden ratio. One may think  of   the equation $x^{n+1}= x^n +1  $ the generalization of another form.  In fact, we can imagine infinitely many forms of such generalizations. Next, we will demonstrate that these two forms are, of course, related. Consider the bilinear transformation
\begin{equation}\label{eq:bilineartransformation}
   w = 1 + \frac{1}{z}.
\end{equation}
The transformation converts the equation (\ref{eq:original}) into a new form of equation:
\begin{equation}\label{eq:transformedequation}
   w^{n+1} - w^n - 1 = 0.
\end{equation}
The roots  $ w $  relate to $ z $  via:
\begin{equation}\label{eq:inversebilinear}
   z = \frac{1}{w - 1}.
\end{equation}
The polynomial equation (\ref{eq:transformedequation}), when $n=1$ is a quadratic equation whose positive root is the golden ratio $\varphi $. These family of equation can be considered as  the generalization of the equation of the golden ratio. Regarding to  the images of some of  geometric figures under this transformation, see  Table 3:

\begin{table}[H]
\centering
\caption{Mapping of regions under the transformation $w(z) = 1 + \frac{1}{z}$}
\label{tab:mapping} 
\begin{tabular}{|l|l|}
\hline
Region in the $z$-plane & Region in the $w$-plane \\ \hline
$|z|=1$                & $|w-1| = 1$            \\ \hline
$|z|>1$                & $|w-1|<1$              \\ \hline
$|z|<1$                & $|w-1|>1$              \\ \hline
$|z+1|=1$              & $\Re(w) = \frac{1}{2}$ \\ \hline
$|z+1|>1$              & $\Re(w) > \frac{1}{2}$ \\ \hline
$|z+1|<1$              & $\Re(w) < \frac{1}{2}$ \\ \hline 
\end{tabular}
\end{table}

\section{Summary of the Results}

\begin{itemize}

\item The positive roots (\ref{eq:original}) form strictly increasing sequence that tends to infinity. The smallest positive root that corresponds to $n=1$ is the golden-ratio $\varphi =\frac{1+\sqrt{5}}{2}$.

    \item A possible asymptotic approximation of the sequence of positive roots is given by
    $$ z_n \approx \frac{n}{\ln (n)-\ln(\ln(n) )}.$$

    \item The negative roots of (\ref{eq:original}), which exist for odd values of $n$, form strictly increasing sequence that tends to the limit $-0.5 $. The limit  $-0.5 $ is the least upper bound  of the sequence.
    The smallest negative root that corresponds to $n=1$ is the negative reciprocal golden-ratio $- \frac{1}{\varphi} =\frac{1- \sqrt{5}}{2}$. Thus, the negative solutions of the equation (\ref{eq:original}), lie inside the interval $[\frac{1-\sqrt{5}}{2},  -\frac{1}{2})$ for all odd $n \in \mathbb{N} $.

   \item  A possible  asymptotic approximation of the sequence of negative roots is given by
    $$ z_n \approx -0. 5 - \frac{\ln(2)}{4n}. $$

  \item The non-real complex roots of the equation (\ref{eq:original}) are located inside and on the portion of  the  region common  to the interiors of the unit circles $|z|=1$ and $|z+1|=1 $. In particular the all the non-real complex roots and the negative roots are located on or to the left of the vertical line $x=-1/2 $. This is a plano-convex lens shaped region $R_5$ which is shown Figure 7.

  \item The non-real complex roots have an asymptotic approximation given by:

  $$ z_n  \approx e^{\frac{2 \pi }{3}} \left(1 + \frac{ \pi \sqrt{3}}{n}   -\frac{ \pi i}{3n} \right).  $$

\end{itemize}

 Further investigation of this problem and related problems, may include the following  activities:

\begin{itemize}

 \item \textbf{The refinement of asymptotic approximations.} You may   consider the refinement of the asymptotic approximation by considering $\delta =  O (\frac{1}{n^k})$ for some $k \in \mathbb{N } , k > 1 $.

   \item  \textbf{Computational complexity }:You may develop methods for calculating the roots of an equation, including their computational complexity, or explore the application of numerical methods.

   \item \textbf{Companion problems}: You may Study some  equations similar to the current problems like $z^{n+1}-(1+z)-1=0$, which is the image reflection along the vertical line $x=-1/2$. Other possible problems may include analysis of  polynomial equations of the form
       $$z^n = \pm (1+z) ^{n+1}, \quad    z^n= \pm(z-1)^{n+1}, \quad  z^{n+1}= -(1+z)^n . $$

 \item  \textbf{Application to other problems}: As mentioned in the introduction section of this paper, the analysis of equation (\ref{eq:original}) can be extended to related difference or differential equations, which may facilitate modeling certain physical problems that utilize the current analysis.
   \end{itemize}

\section*{Statements of Declarations:}

\subsection*{Conflict of Interests}

The author declare that there is no conflict of interests regarding the publication of this paper.

\subsection*{Acknowledgment}

The author is thankful to the anonymous reviewers for their constructive and valuable suggestions.

\subsection*{Author's contribution}

The corresponding author is  the sole contributor of the whole content of this work.

\subsection*{Data Availability}

There are no external data used in this paper other than the reference materials.

\subsection*{Funding}

 This Research work is not funded by any institution or individual.


\end{document}